\documentclass[11pt]{amsart}
\usepackage{verbatim}
\usepackage{rotating} 
\usepackage{amssymb}
\usepackage[bookmarks,colorlinks,breaklinks]{hyperref}  
\hypersetup{linkcolor=blue,citecolor=blue,filecolor=blue,urlcolor=blue} 
\usepackage{mathrsfs,amsfonts,amsmath,amssymb,epsfig,amscd,xy,amsthm,pb-diagram} 

\newtheorem{theorem}{Theorem}[section]
\newtheorem{proposition}[theorem]{Proposition}

\newtheorem{lemma}[theorem]{Lemma}

\newtheorem{corollary}[theorem]{Corollary}

\newtheorem{definition}[theorem]{Definition}

\theoremstyle{plain}

\theoremstyle{remark}

\newtheorem{remark}[theorem]{Remark}

\newcommand{\tb}{\tilde{b}}
\newcommand{\tc}{\tilde{c}}
\newcommand{\C}{{\mathbb C}}

\newcommand{\Q}{{\mathbb Q}}

\newcommand{\N}{{\mathbb N}}

\newcommand{\fp}{\mathfrak p}

\newcommand{\Gm}{\mathbb{G}_{\text{m}}}
\newcommand{\Gmn}{\mathbb{G}_{\text{m}}^{n}}

\newcommand{\Qbar}{\bar{\Q}}

\newcommand{\Ebar}{\bar{E}}

\newcommand{\bP}{{\mathbb P}}

\newcommand{\bA}{{\mathbb A}}

\newcommand{\lra}{\longrightarrow}

\author{D.~Ghioca}
\address{
Dragos Ghioca\\
Department of Mathematics\\
University of British Columbia\\
Vancouver, BC V6T 1Z2\\
Canada
}
\email{dghioca@math.ubc.ca}

\author{K.~D.~Nguyen}
\address{
Khoa D.~Nguyen \\
Department of Mathematics\\
University of British Columbia\\
And Pacific Institute for The Mathematical Sciences\\ 
Vancouver, BC V6T 1Z2, Canada}
\email{dknguyen@math.ubc.ca}
\urladdr{www.math.ubc.ca/\~{}dknguyen}


\begin{document}
	\title[Uniform bounds for periods and application]{Dynamics of split polynomial maps: uniform bounds for periods and applications}
	
	
	\begin{abstract}
	Let $K$ be an algebraically closed field of
	characteristic 0. Following Medvedev-Scanlon, a polynomial 
	of degree $\delta\geq 2$ is said to be disintegrated
	if it is not linearly conjugate to
	$x^{\delta}$ or $\pm T_{\delta}(x)$ where $T_{\delta}(x)$ is the Chebyshev 
	polynomial of degree $\delta$.
	Let $d$ and $n$ be integers greater than 1, we 
	prove that there exists an effectively computable constant 
	$c(d,n)$ depending only
	on $d$ and $n$ such that the following holds. Let 
	$f_1,\ldots,f_n\in K[x]$ be disintegrated polynomials
	of degree at most $d$ and let 
	$\varphi=f_1\times\cdots\times f_n$ be the 
	induced coordinate-wise self-map of 
	$\bA^n_K$. Then the period of every irreducible $\varphi$-periodic
	subvariety of $\bA^n_K$ with non-constant projection
	to each factor $\bA^1_K$
	is at most
	$c(d,n)$. As an immediate application, we prove an
	instance
	 of the dynamical Mordell-Lang problem
	following recent work of Xie. The main technical
	ingredients are Medvedev-Scanlon classification of
	invariant subvarieties together with classical and more
	recent results in Ritt's theory of polynomial 
	decomposition.
	\end{abstract}

	\maketitle{}
	
 \section{Introduction}\label{sec:introduction}
  Throughout this paper, let $\N$ denote the set of positive
  integers, $\N_0:=\N\cup\{0\}$, and let $K$ be an algebraically closed
  field of characteristic 0. For a map $\mu$ from a set to
  itself and for $m\in\N$, we let
  $\mu^{\circ m}$ denote the $m$-fold iterate
  $\mu\circ\ldots\circ\mu$. For every 
  polynomial $f(x)\in K[x]$, the notations $f^{\circ m}$,
  $f^{\circ m}(x)$, and $f(x)^{\circ m}$ 
  have the same meaning which is the $m$-fold 
  iterate $f\circ\ldots\circ f$.  
  
  Let $\delta\geq 2$ be an integer,
  the Chebyshev polynomial $T_\delta(x)\in K[x]$ is the
  polynomial of degree $\delta$ satisfying $\displaystyle 
  T_\delta\left(x+\frac{1}{x}\right)=x^\delta+\frac{1}{x^\delta}$. Following the terminology in Medvedev-Scanlon \cite{MedSca}, we say that a polynomial 
  $f(x)\in K[x]$ is
  \emph{disintegrated} 
  if it has degree $\delta\geq 2$
  and it is not linearly conjugate to
  $x^\delta$ or $\pm T_\delta(x)$. Let $n\in\N$ 
  and let $f_1,\ldots,f_n\in K[x]$ be polynomials
  of degree at least 2 . Let $\varphi=f_1\times\cdots\times f_n$
  be the induced coordinate-wise self-map of 
  $\bA^n_K$. An irreducible subvariety
  $V$ of $\bA^n_K$ is said to be $\varphi$-periodic
  if there exists $N\in\N$ satisfying
  $\varphi^{\circ N}(V)=V$; the smallest such $N$ is called
  the period  (or more precisely $\varphi$-period) of $V$.
  According to \cite[Theorem~2.30]{MedSca},
  to study the algebraic dynamics of $\varphi$,
  it suffices to study two cases: the case when
  none of the $f_i$'s are disintegrated which reduces
  to the geometry of $\Gmn$ (see, for instance, 
  \cite[Chapter~3]{BG}) and the case when all the $f_i$'s are
  disintegrated. In the first case, we note that the period 
  under the multiplication-by-$d$ map
  of a subvariety of $\Gm^2$
  can be as large as possible. For instance, the torsion 
  translate
  defined by $x=\zeta y$ where $\zeta$ is a primitive
   $(d^N-1)$-th root of unity has period $N$ under the map
   $(x,y)\mapsto (x^d,y^d)$. On the other hand, we obtain
   the following result in the latter case:
  
 \begin{theorem}\label{thm:intro1}
 	Let $n\geq 1$ and $d\geq 2$ be integers, then there is a 
 	an effectively computable
 	constant
 	$c(d,n)$ depending only on $d$ and $n$ such that the 
 	following holds.
 	 Let $f_1,\ldots,f_n\in K[x]$
 	be disintegrated polynomials of degree at 
 	most $d$.
 	Let $\varphi:=f_1\times\cdots\times f_n$
 	be the induced coordinate-wise self-map on $\bA^n_K$. 
 	Let $V$ be an irreducible $\varphi$-periodic subvariety
    of $\bA^n_K$ such that the projection from
    $V$ to each factor $\bA^1_K$ is non-constant.
 	Then the period of $V$ is at most $c(d,n)$. 
 \end{theorem}

 In Theorem~\ref{thm:intro1}, the condition that the projection 
 to each factor $\bA^1_K$
 is non-constant is necessary. For instance, let $\zeta_1$
 be periodic under $f_1$, then the $\varphi$-period
 of $\{\zeta_1\}\times \bA^{n-1}_K$ is the $f_1$-period
 of $\zeta_1$ which can be arbitrarily large. Following
 the proof (see Section~\ref{sec:proof main}), the constant
 $c(d,2)$ can be taken to be $\displaystyle\frac{d^{2d^4}}{2}$ and for $n\geq 3$, 
 $c(d,n)$ can
 be defined recursively in an explicit way. It is 
 conceivable that such constants  are not best possible; the goal of our theorem is to show that they depend only on $d$ and $n$, and they are independent on the actual coefficients of the polynomials $f_i$. The proof of Theorem~\ref{thm:intro1} uses the classification of $\varphi$-period
 subvarieties by Medvedev-Scanlon \cite{MedSca} and a uniform
 bound involving semiconjugacy relation among disintegrated
 polynomials that might be of independent interest (see Section~\ref{sec:semiconjugacy}, in particular Theorem~\ref{thm:semiconjugate}).

 As an  application of our Theorem~\ref{thm:intro1}, we
 prove the following instance (see Theorem~\ref{thm:DML}) of the Dynamical Mordell-Lang
 Conjecture following recent work of Junyi Xie \cite{Xie-2015}.
 For a survey on this conjecture, we refer
 the readers to the upcoming book of Jason Bell, Thomas Tucker,
 and the first author \cite{BGT-book}. By an arithmetic progression, we mean a set of the form $\{a+bk:\ k\in\N_0\}$
 for some $a,b\in\N_0$. Note that this definition
 also includes singletons (when $b=0$). We have the following:
  
 \begin{theorem}\label{thm:DML}
 Let $F_1,\dots, F_n\in K[x]$ and 
 let $\Phi=F_1\times\cdots\times F_n$ be the induced 
 coordinate-wise self-map of $\bA^n_K$. Let $C\subset \bA^n_K$ 
 be a curve, and let $\alpha\in \bA^n(K)$. Then 
 the set $\{m\in\N\colon \Phi^{\circ m}(\alpha)\in C\}$ is a 
 finite union of arithmetic progressions.
 \end{theorem} 
  
 Theorem~\ref{thm:DML} when $K=\Qbar$ was obtained by 
 Xie \cite[Theorem~0.3]{Xie-2015}; we note that \cite{Xie-2015} is a veritable tour de force---an almost 100 pages long paper which proves the Dynamical Mordell-Lang Conjecture for plane curves. Theorem~\ref{thm:DML}   
 removes the condition $K=\Qbar$ and it is proven by combining Xie's 
 result together
 with standard specialization arguments using 
 the uniform bound from Theorem~\ref{thm:intro1}. During the 
 preparation of this paper, we learned from Xie that after
 some further steps and a more careful 
 analysis of the arguments in \cite{Xie-2015},
 he is able to extend his results over
 any field.  However, Xie's extension of his \cite[Theorem~0.3]{Xie-2015} from $K=\Qbar$ to an arbitrary field $K$ uses completely different arguments than the ones we employ, since he does not prove a uniform bound for the period of a variety as in our Theorem~\ref{thm:intro1}. Instead, Xie told us that he needs to re-write his proof from \cite{Xie-2015} in a more general setup since his arguments do not easily extend from $\Qbar$ to an arbitrary field using specialization theorems. On the other hand, the main result of our paper (Theorem~\ref{thm:intro1}) provides a very direct route to proving Theorem~\ref{thm:DML} using specialization techniques.
  
 The organization of this paper is as follows. In the next
 section, we present a theorem involving semiconjugate 
 polynomials (see Theorem~\ref{thm:semiconjugate}). In order
 to prove this theorem, we need some classical and new
 results in Ritt's theory of polynomial decomposition
 as proven in the seminal paper \cite{ZieveMuller-2008}. Then we introduce
 the Medvedev-Scanlon classification of periodic subvarieties 
 under split polynomial maps \cite{MedSca} and give the proof of
 Theorem~\ref{thm:intro1}. The proof of Theorem~\ref{thm:DML}
 is given in the last section. 
  
  {\bf Acknowledgments.} We are grateful to Alice Medvedev,
  Fedor Pakovich, Thomas Scanlon, and Junyi Xie  
	for useful discussions. The first author is partially supported by NSERC and the second author is partially supported
	by a UBC-PIMS postdoctoral fellowship.
  
  \section{Uniform Bounds for Semiconjugate Polynomials}\label{sec:semiconjugacy}
  \begin{definition}
   Let $f,\eta \in K[x]$ be polynomials of degree at least 2. The polynomial $\eta$ is said to be semiconjugate to $f$
   if there is a non-constant polynomial $p(x)\in K[x]$ such
   that $f\circ p=p\circ \eta$. 
  \end{definition}
  
  For recent results on semiconjugate polynomials, we
  refer the readers to Pakovich's paper \cite{Pakovich-2015}.
  It is not difficult to prove that if $\eta$ is
  semiconjugate to $f$ then $f$ is disintegrated if
  and only if $\eta$ is disintegrated
  (see, for instance, \cite[Theorem~4.4]{Pakovich-2015}). We
  define the relation $\approx$ in the set of disintegrated
  polynomials:
  \begin{definition}\label{def:equiv}
  For any disintegrated polynomials $f$ and $g$ in $K[x]$,
  we write $f\approx g$ if there exist $N\in\N$
  such that there is a polynomial $\eta$ that is 
  semiconjugate
  to both $f^{\circ N}$ and $g^{\circ N}$.
  \end{definition}
  According to \cite[Corollary~2.35]{MedSca}, $f\approx g$
  if and only if the self-map $(x,y)\mapsto (f(x),g(y))$
  of $\bA^2$ admits a periodic curve $C$ having non-constant 
  projection to each factor $\bA^1$. Then it is not difficult  
  to
  show that $\approx$ is an \emph{equivalence relation}
  in the set of disintegrated polynomials (see \cite[Section~7.1]{DBHC}). As explained in
  \cite[Section~7.1]{DBHC}, if $f_1,\ldots,f_m$ are
  in the same equivalence class under $\approx$, then there is
  $N\in\N$ such that there exists a common polynomial
  $\eta$ that is semiconjugate to each of the 
  polynomial $f_i^{\circ N}$.
  Besides the Medvedev-Scanlon classification, the key result to the proof
  of Theorem~\ref{thm:intro1} is the
  following:
  
  \begin{theorem}\label{thm:semiconjugate}
  Let $d\geq 2$ be an integer. Let $f$ and $g$ be 
  disintegrated 
  polynomials in $K[x]$ of degree at most $d$ 
  satisfying $f\approx g$.
  Then there exists $N\leq 2d^4$ and a common polynomial
  $\eta\in K[x]$ that is semiconjugate
  to both  
  $f^{\circ N}$ and $g^{\circ N}$.
  \end{theorem}
	
	By a simple induction argument, we have the
	following extension
	of Theorem~\ref{thm:semiconjugate}:
	\begin{corollary}\label{cor:new semiconjugate}
	Let $d$ and $n$ be integers at least 2, then there
	exists an effectively computable constant $c_1(d,n)$
	depending only on $d$ and $n$ 
	such that the following holds. Let $f_1,\ldots,f_n$ be
	disintegrated polynomials of degree at most $d$ 
	belonging to the same
	equivalence class under $\approx$. Then there exist 
	a positive integer $N\leq c_1(d,n)$ and 
	a polynomial $\theta\in K[x]$
	that is semiconjugate to $f_i^{\circ N}$
	for $1\leq i\leq n$.
	\end{corollary}

	The proof of Theorem~\ref{thm:semiconjugate}
	uses classical and more recent results
	in Ritt's theory of polynomial
	decomposition as presented in the paper
	by M\"uller and Zieve \cite{ZieveMuller-2008}
	and Inou \cite{Inou}. For another
	remarkable uniform bound in polynomial
	decomposition, we refer
	the readers to \cite[Theorem~1.4]{ZieveMuller-2008}
	which also has applications to
	the Dynamical Mordell-Lang Conjecture \cite{GTZ12}.

	We now introduce 
	an interesting result of Pakovich and a consequence of his result
	and of Corollary~\ref{cor:new semiconjugate}. The rest of
	this section is not needed for the proof of Theorem~\ref{thm:semiconjugate}. We have \cite[Theorem~1.4]{Pakovich-2015}:
	
	\begin{theorem}[Pakovich]\label{thm:Pakovich}
	Let $n\geq 2$ be an integer, then there exists a constant
	$c_3(n)$ depending only on $n$ such that the following 
	holds. Let $B\in K[x]$ be a disintegrated polynomial of 
	degree
	$n$, let $\mathcal{F}(B)$ denote the set
	of polynomials $P(x)\in K[x]$ such that 
	$P\circ S=S\circ B$ for some non-constant $S(x)\in K[x]$
	(i.e. $B$ is semiconjugate to $P$).	
	Then there is a such a pair $(P,S)$ with 
	$\deg(S)\leq c_3(n)$ that is 
	universal in the following sense. For every 
	$Q(x)\in \mathcal{F}(B)$, there exist
	non-constant $S_1(x),S_2(x)\in K[x]$
	such that: $S=S_2\circ S_1$, $P\circ S_2=S_2\circ Q$,
	and $Q\circ S_1=S_1\circ B$. 
	\end{theorem}
	
	By combining Corollary~\ref{cor:new semiconjugate} 
	and Theorem~\ref{thm:Pakovich}, we have
	the following:
	\begin{corollary}
		Let $d$ and $n$ be integers at least 2, then there 
		exists a constant $c_4(d,n)$ depending only on $d$
		and $n$ such that the following holds. Let $f_1,\ldots,f_n$ be disintegrated polynomials of degree at most $d$ belonging to the same equivalence class under $\approx$. Then
		there exist a positive integer $N\leq c_1(d,n)$,
		polynomials $\theta$ and $P$ of degree
		$d^N$, a non-constant polynomial $S$, 
		and non-constant polynomials
		$S_{i,2},S_{i,1}$ of degree at most $c_4(d,n)$
		for $1\leq i\leq n$ such that:
		$f_i^{\circ N}\circ  S_{i,1}=S_{i,1}\circ \theta$,
		$P\circ S_{i,2}=S_{i,2}\circ f_i^{\circ N}$,
		and $S=S_{i,2}\circ S_{i,1}$ for
		$1\leq i\leq n$. 
	\end{corollary}
	\begin{proof}
		By Corollary~\ref{cor:new semiconjugate}, there
		exist a positive integer $N\leq c_1(d,n)$
		and a polynomial $\theta$ such that 
		$f_i^{\circ N}\in \mathcal{F}(\theta)$
		for $1\leq i\leq n$. The desired result now
		follows from Theorem~\ref{thm:Pakovich}.
	\end{proof}
	
	\section{Results in Ritt's theory of polynomial
	decomposition}
	The following result by Engstrom \cite{Engstrom} will be
	used repeatedly in this paper: 
	
	\begin{lemma}\label{lem:Engstrom}
		If $a,b,c,d\in K[x]\setminus K$ satisfy
		$a\circ b=c\circ d$ then there exist
		$\hat{a},\hat{b},\hat{c},\hat{d},g,h\in\C[x]$
		such that:
		\begin{itemize}
			\item $g\circ \hat{a}=a$, $g\circ \hat{c}=c$, $\deg(g)=\gcd(\deg(a),\deg(c))$.
			\item $\hat{b}\circ h=b$, $\hat{d}\circ h=d$, $\deg(h)=\gcd(\deg(b),\deg(d))$.
			\item $\hat{a}\circ\hat{b}=\hat{c}\circ \hat{d}$.
		\end{itemize}
		Consequently, if $\deg(a)=\deg(c)$ then there
		is a linear polynomial $\ell$ such that
		$a=c\circ \ell$ and $b=\ell^{-1}\circ d$.
	\end{lemma}

	Ritt's second
	theorem \cite{Ritt1922} (also see
	\cite{Tortrat} or \cite[Appendix]{ZieveMuller-2008}) 
	then allows us to 
	study the equation $\hat{a}\circ \hat{b}=\hat{c}\circ\hat{d}$ under the condition $\gcd(\deg(\hat{a}),\deg(\hat{c}))=\gcd(\deg(\hat{b}),\deg(\hat{d}))=1$. For a version of
	Ritt's theorem including the case of positive
	characteristic, we refer the readers
	to Schinzel's book \cite{Schinzel-book}
	following an earlier paper of Zannier 
	\cite{Zannier-Ritt2nd}. We have the following
	immediate consequence of Ritt's theorem
	for semiconjugacy functional equations:
	\begin{lemma}\label{lem:Inou}
	Let $f$ be a disintegrated polynomial of degree $\delta$.
	Let $p$ and $\eta$ be non-constant polynomials such
	that $f\circ p=p\circ \eta$ and $\gcd(\delta,\deg(p))=1$.
	Then there exist linear polynomials 
	$\ell_1,\ell_2,\ell_3$,
	positive integers $b$, $c$, and 
	a non-constant polynomial $P$ such that the following hold:
	\begin{itemize}
		\item $P(0)\neq 0$ and $c\equiv b$ modulo $\delta$.
		\item $\ell_1\circ f\circ\ell_1^{-1}=x^cP(x)^b$.
		\item $\ell_1\circ p\circ \ell_2^{-1}=x^b$.
		\item $\ell_2\circ \eta\circ \ell_2^{-1}=x^cP(x^b)$.
	\end{itemize}
	\end{lemma}	
	\begin{proof}
	See \cite[Appendix A]{Inou}.
	\end{proof}
	
	Following \cite[Section~3]{ZieveMuller-2008}, we have
	the following definitions:
	\begin{definition}\label{def:equivalent}
	Two non-constant polynomials $A$ and $B$ are said to
	be equivalent if there are linear polynomials 
	$L_1$ and $L_2$ such that
	$L_2\circ A\circ L_1=B$. Write $\delta=\deg(A)$. The
	polynomial $A$ is said to be \emph{cyclic}
	if $A$ is equivalent to $x^{\delta}$ and $\delta\geq 2$.
	The polynomial $A$ is said to be \emph{dihedral}
	if $A$ is equivalent to $T_{\delta}(x)$ and $\delta\geq 3$.
	\end{definition}
	
	\begin{definition}\label{def:Gamma}
	Let $A$ be a non-constant polynomial. The group $\Gamma(A)$
	is defined to be the group (under composition) of linear 
	polynomials $\ell$
	such that $A\circ\ell=L\circ A$ for some linear
	polynomial $L$.
	\end{definition}
	
	We have the following lemmas:
	\begin{lemma}\label{lem:f^4}
		Let $F$ be a disintegrated polynomial. 
		Then $F^{\circ 2}$ is not cyclic; moreover 
		$F^{\circ 2}$
		is not dihedral if $\deg(F)\geq 3$. Consequently, 
		$f^{\circ 4}$ is neither cyclic nor dihedral for every 
		disintegrated polynomial $f$.
	\end{lemma}
	\begin{proof} 
		Write $\delta=\deg(F)\geq 2$. First assume that $F^{\circ 2}$ is cyclic, hence
		$F^{\circ 2}=L_1\circ x^{\delta^2}\circ L_2$
		for some linear polynomials $L_1$ and $L_2$. By
		Lemma~\ref{lem:Engstrom}, there
		is a linear polynomial $L$ such that:
		$$F=L_1\circ x^{\delta}\circ L=L^{-1}\circ x^{\delta}\circ L_2.$$
		Hence 
		$L\circ L_1\circ x^{\delta}=x^{\delta}\circ L_2\circ L^{-1}$. By
		comparing the coefficients of $x^{\delta-1}$, we
		have that $L_2\circ L^{-1}=\gamma x$,
		hence $L_2=\gamma L$. This gives $F=L^{-1}\circ \gamma^{\delta}x^{\delta}\circ L$ which is linearly conjugate to $\gamma^{\delta}x^{\delta}$, hence  to $x^\delta$, contradiction. 
		
		Now assume that $\delta\geq 3$ and  $F^{\circ 2}$
		is dihedral. By similar arguments, we have
		linear polynomials $L_1,L_2,L$ such that:
		$$F=L_1\circ T_{\delta}\circ L=L^{-1}\circ T_{\delta}\circ L_2.$$  
		Therefore
		$L\circ L_1\circ T_\delta=T_\delta\circ L_2\circ L^{-1}$ which
		yields $L_2\circ L^{-1}=\pm x$; this fact follows using the fact that the coefficient of $x^{\delta -1}$ in $T_\delta$ is $0$ while the coefficient of $x^{\delta-2}$ in $T_{\delta}$
		is nonzero. Thus $L_2=\pm L$, and so, we have
		 $F=L^{-1}\circ T_\delta(\pm x)\circ L$ which
		is linearly conjugate to $\pm T_\delta(x)$,
		contradiction. 
		
		The last assertion follows from applying the previous assertions for $F=f^{\circ 2}$ (note that $f^{\circ 2}$ is also disintegrated). 
	\end{proof}
	
	\begin{lemma}\label{lem:Gamma}
	Let $A$ be a polynomial of degree $\delta\geq 2$. The following hold:
		\begin{itemize}
			\item [(a)] Let $L_1$
			and $L_2$ be linear polynomials. Then
			the map $\ell\mapsto L_2^{-1}\circ\ell\circ L_2$
			is an isomorphism
			from $\Gamma(A)$ to $\Gamma(L_1\circ A\circ L_2)$. 
		
			\item [(b)] For every positive integer
			$n$, $\Gamma(f^{\circ n})$ is a subgroup of $\Gamma$.
		
			\item [(c)] $\Gamma(A)$ is infinite
			if and only if $A$ is cyclic.
			
			\item [(d)] Assume that $\Gamma(A)$ is
			finite of order $n$. Then $A$
			is equivalent to a polynomial of the form
			$x^sP(x^n)$ where $P$ is not a polynomial
			in $x^j$ for any $j\geq 2$. Consequently, $\Gamma$
			is cyclic of order $n\leq \delta/2$.
			
			\item [(e)] If $A$ is equivalent
			to a polynomial of the form $x^sP(x^n)$ with
			$s>0$, $n\geq 2$ and non-constant 
			$P(x)\in K[x]\setminus xK[x]$, then $A$ is not
			cyclic.
		\end{itemize}
	\end{lemma}
	\begin{proof}
	Part (a) is straightforward. For parts (c) and (d),
	see \cite[Lemma~3.17]{ZieveMuller-2008}. For part (b), let
	$\ell\in \Gamma(f^{\circ n})$, 
	therefore $f^{\circ n}\circ \ell=L\circ f^{\circ n}$ for 
	some linear polynomial
	$L$. Then Lemma~\ref{lem:Engstrom}
	gives that $f\circ \ell=L_1\circ f$
	for some linear $L_1$. This gives $\ell\in \Gamma(f)$
	and proves part (b).
	
	For part (e), assume that there are linear polynomials
	$L_1$ and $L_2$ such that $x^sP(x^n)\circ L_1=L_2(x^\delta)$.
	By comparing the coefficients of
	$x^{\delta-1}$ and the constant coefficients, we have that 
	$L_1(x)=ax$ and $L_2(x)=bx$. Since $P(x)\notin xK[x]$,
	we get a contradiction.
	\end{proof}
	
	\begin{lemma}\label{lem:f^NLg^N}
	Let $f$ and $g$ be polynomials of degree $\delta\geq 2$
	both of which are not cyclic. Assume there
	exist a linear polynomial $L$ and a positive integer
	$n\geq \frac{\delta+1}{2}$ such that 
	$f^{\circ n}=L\circ g^{\circ n}$. Then 
	there exist a linear polynomial $\ell$ and a positive 
	integer
	$N\leq \frac{\delta}{2}$ 
	such that 
	$f^{\circ N}=(\ell\circ g\circ \ell^{-1})^{\circ N}$.
	\end{lemma}
	\begin{proof}
		Applying Lemma~\ref{lem:Engstrom} repeatedly,
		there exist linear polynomials
		$L_0(x):=x,L_1,\ldots,L_{n-1}, L_n:=L$ such that:
		$$f=L_1\circ g=L_2\circ g\circ L_1^{-1}=L_3\circ g\circ L_2^{-1}=\ldots=L_{n-1}\circ g\circ L_{n-2}^{-1}=L\circ g\circ L_{n-1}^{-1}.$$
		From $L_1\circ g=L_{i+1}\circ g\circ L_{i}^{-1}$,
		we have that $L_i\in \Gamma(g)$ for every 
		$0\leq i\leq n-1$. Since $|\Gamma(g)|\leq \delta/2<n$, 
		there exist $0\leq i<j\leq \delta/2$ such that 
		$L_i=L_j$. We now have:
		$$f^{\circ (j-i)}=(L_{j}\circ g\circ L_{j-1}^{-1})\circ\ldots\circ (L_{i+1}\circ g\circ L_i^{-1})=L_j\circ g^{\circ (j-i)}\circ L_i^{-1}.$$
		Since $L_j=L_i$, this finishes the proof.
	\end{proof}
	
	\begin{lemma}\label{lem:decomposing xsPn}
	Let $A,B,P\in K[x]\setminus K$ and let
	$s$ and $n$ be coprime positive integers. If $A\circ B=x^sP(x)^n$ then
	$a=x^jP_1(x)^n\circ \ell$ 
	and $b=\ell^{-1}\circ x^kP_2(x)^n$
	for some $j,k>0$ satisfying $\gcd(j,n)=\gcd(k,n)=1$,
	and some $P_1,P_2,\ell\in K[x]$ with $\ell$ linear. 
	\end{lemma}
	\begin{proof}
	 This is proved in \cite[Lemma~3.11]{ZieveMuller-2008}.
	\end{proof}
	
	\begin{lemma}\label{lem:linear xsPn}
	Let $s>0$ and $n\geq 2$ be coprime integers. Let $P\in K[x]\setminus xK[x]$, and suppose that $f(x):=x^sP(x)^n$
	is neither cyclic nor dihedral. Then for any linear
	polynomials $\ell_1,\ell_2$ satisfying
	$\ell_1\circ f\circ \ell_2=x^{\tilde{s}}P_1(x)^{\tilde{n}}$
	with coprime integers $\tilde{s}>0$, $\tilde{n}\geq 2$,
	and polynomial $P_1(x)\in K[x]\setminus xK[x]$,
	we have that $\ell_1(x)=\gamma x$. 
	\end{lemma}
	\begin{proof}
	This is proved in \cite[Lemma~3.21]{ZieveMuller-2008}.
	\end{proof}
	
	We have the following:
	\begin{lemma}\label{lem:f^4 xsPn}
		Let $f$ be a disintegrated polynomial 
		and let $N\geq 4$. Assume there are
		coprime integers $s>0$, $n>1$, and 
		non-constant polynomial
		$P(x)\in K[x]\setminus xK[x]$
		such that $f^{\circ N}=x^sP(x)^n$.
		Then there exist a positive integer $\tilde{s}$
		and a non-constant polynomial 
		$P_1(x)\in K[x]\setminus xK[x]$
		such that $f^{\circ 4}=x^{\tilde{s}}P_1(x)^n$.
	\end{lemma}
	\begin{proof}
		By Lemma~\ref{lem:f^4}, $f^{\circ 4}$ is neither cyclic nor 
		dihedral. By applying Lemma~\ref{lem:decomposing xsPn}
		to the identity 
		$f^{\circ 4}\circ f^{\circ (N-4)}=f^{\circ (N-4)}\circ f^{\circ 4}=x^sP(x)^n$, 
		there exist
		linear polynomials $L_3,L_4$,
		positive integers $j,k$ each of which is coprime to
		$n$, and non-constant polynomials
		$P_1(x),P_2(x)\in K[x]\setminus xK[x]$
		such that:
		$$f^{\circ 4}=x^{j}P_1(s)^n\circ L_3=L_4^{-1}\circ x^{k}P_2(x)^n.$$ 
		Since $f^{\circ 4}$ is neither cyclic nor dihedral, $x^jP_1(s)^n$ is neither cyclic nor dihedral.
		From $L_4\circ x^jP_1(x)^n\circ L_3=x^{k}P_2(x)^n$
		and Lemma~\ref{lem:linear xsPn}, we have that 
		$L_4=\gamma x$ for some $\gamma\in K^*$. Hence $f^{\circ 4}=L_4^{-1}\circ x^kP_2(x)^n$
		has the desired form.
	\end{proof}
	
	The following result is well-known in Ritt's theory:
	\begin{lemma}\label{lem:2.3 in IMRN}
		Let $f\in K[x]$ be a disintegrated polynomial.
			\begin{itemize}
			\item [(a)] If $g\in K[x]$ has degree 
			at least 2 such that 
			$g$ commutes with an iterate of $f$ then $g$ 
			and $f$ have a common iterate.
			\item [(b)] Let $M(f^\infty)$
			denote the collection of all linear polynomials
			commuting with an iterate of $f$. Then 
			$M(f^\infty)$ is a finite cyclic group under 
			composition.
			\item [(c)] Let $\tilde{f}\in F[X]$ be a polynomial
			of lowest degree at least 2 such that $\tilde{f}$
			commutes with an iterate of $f$. Then there exists
			$e>0$ relatively prime to
			the order of $M(f^\infty)$ such that 
			$\tilde{f}\circ L=L^{\circ e}\circ \tilde{f}$
			for every $L\in M(f^\infty)$.
			\item [(d)] $\left\{\tilde{f}^{\circ m}\circ L:\ m\geq 0,\ L\in M(f^\infty)\right\}=\left\{L\circ\tilde{f}^{\circ m}:\ m\geq 0,\ L\in M(f^\infty)\right\}$, and these sets consist of exactly all
			the polynomials $g$ commuting with an iterate of $f$.
			\item [(e)] The order of $M(f^\infty)$
			is at most $\deg(\tilde{f})/2$.				
			\end{itemize}	
	\end{lemma}
	\begin{proof}
		For parts (a), (b), (c), and (d), see 
		Proposition~2.3 and Remark~2.4 in 
		\cite{Nguyen-IMRN2015}. For part (e), if
		$\tilde{f}$ is not cyclic, we
		have $M(f^\infty)=M(\tilde{f}^{\infty})$
		is a subgroup of $\Gamma(\tilde{f})$
		which has at most $\deg(\tilde{f})/2$ elements
		by Lemma~\ref{lem:Gamma}. If $\tilde{f}$
		is cyclic, by taking a linear conjugation if
		necessary,
		we may assume that $\tilde{f}=x^m+c$ with $c\neq 0$. 
		For every $n\geq 1$, $\tilde{f}^{\circ n}$ has the form 
		$x^{m^n}+ax^{m^n-m}+\ldots$
		for some $a\neq 0$. Hence $M(\tilde{f}^\infty)$
		contains linear polynomials of the form $\zeta x$
		where $\zeta$ is an $m$-th root of 
		unity. Since $\tilde{f}(\zeta x)=\tilde{f}$, we 
		conclude that $M(\tilde{f}^\infty)$ is trivial.
	\end{proof}
	
	\begin{lemma}\label{lem:commute everything}
	Let $f\in K[x]$ be a disintegrated polynomial of degree $\delta$.
	Then there exists $n\leq \delta/2$ such that the following hold. If $g$
	is a non-constant polynomial
	commuting with an iterate of $f$ then
	$f^{\circ n}$ commutes with $g$.  
	\end{lemma}
	\begin{proof}
		Let $\tilde{f}$, $M(f^\infty)$, and 
		$e$ be as in the statement of 
		Lemma~\ref{lem:2.3 in IMRN}. By part (d)
		of
		Lemma~\ref{lem:2.3 in IMRN}, it suffices to
		show that there is $n\leq \deg(\tilde{f})/2\leq \delta/2$
		such that $\tilde{f}^{\circ n}$
		commutes with every element in $M(f^\infty)$.
		By part (c) of Lemma~\ref{lem:2.3 in IMRN}, 
		$\tilde{f}^{\circ n}\circ L=L^{\circ e^n}\circ \tilde{f}^{\circ n}$ for
		every $L\in M(f^\infty)$.
		We simply choose $n$ such that 
		$e^n\equiv 1$ modulo $|M(f^\infty)|$. Since
		$|M(f^\infty)|\leq \deg(\tilde{f})/2$,
		such an $n$ could be chosen to be at most 
		$\deg(\tilde{f})/2$.
	\end{proof}
	
	\section{Proof of Theorem~\ref{thm:semiconjugate}
	and Corollary~\ref{cor:new semiconjugate}}\label{sec:proof semiconjugate}
	We start with the following easy lemma:
	\begin{lemma}\label{lem:AB disintegrated}
	Let $A$ and $B$ be non-constant polynomials. Assume
	that $A\circ B$ is disintegrated, then 
	$B\circ A$ is disintegrated.
	\end{lemma}
	\begin{proof}
	Let $\alpha=\deg(A)$ and $\beta=\deg(B)$. 
	Assume that $B\circ A$ is linearly conjugate to
	$x^{\alpha\beta}$, namely $B\circ A=L\circ x^{\alpha\beta}\circ L^{-1}$. By Lemma~\ref{lem:Engstrom}, there is a
	linear $L_1$ such that $B=L\circ x^{\beta}\circ L_1$,
	and $A=L_1^{-1}\circ x^{\alpha}\circ L^{-1}$.
	Then $A\circ B=L^{-1}\circ x^{\alpha\beta}\circ L$,
	contradiction. The case that $B\circ A$ is linearly
	conjugate to $T_{\alpha\beta}(x)$ is
	settled similarly.
	\end{proof}
	
	\begin{proof}[Proof of Theorem~\ref{thm:semiconjugate}]  
	Assume there is a counter-example
	to Theorem~\ref{thm:semiconjugate}. Among all such
	counter-examples, let $(f,g)$
	satisfy the following properties: 
	\begin{itemize}
	\item [(a)] For every $N\leq 2d^4$,
	there does not exist a polynomial
	that is semiconjugate to both $f^{\circ N}$ and
	$g^{\circ N}$.
	\item [(b)] There exist non-constant polynomials
	$p,q,\theta\in K[x]$ and $n\in\N$ (necessarily
	greater than $2d^4$) such that
	$f^{\circ n}\circ p=p\circ\theta$ and
	$g^{\circ n}\circ q=q\circ \theta$. 
	\item [(c)] The number 
	$\deg(p)+\deg(q)$ is minimal among all counter-examples
	and all the quadruples $(\theta,p,q,n)$. 
	\end{itemize}
	
	Write $\delta=\deg(f)=\deg(g)$. Let $D=\gcd(\delta,\deg(p))$. 

Assume first that 
	$D>1$. Then by Lemma~\ref{lem:Engstrom}, 
	we can write $f=A\circ f_0$, 
	$p=A\circ p_0$ with $\deg(A)=D$ and:
	$$f_0\circ f^{\circ (n-1)}\circ A\circ p_0=p_0\circ \theta.$$
	Hence $(f_0\circ A)^{\circ n}\circ p_0=p_0\circ \theta$.
	By Lemma~\ref{lem:AB disintegrated}, $f_0\circ A$
	remains disintegrated. 
	Because of
	$\deg(p_0)<\deg(p)$
	and the minimality of $\deg(p)+\deg(q)$, the pair
	$(f_0\circ A, g)$ is not a counter-example to 
	Theorem~\ref{thm:semiconjugate}. So there is $N\leq 2d^4$
	such that there exists a common polynomial $B$ that is 
	semiconjugate to both $(f_0\circ A)^{\circ N}$ and $g^{\circ N}$. This implies
	that $B$ is semiconjugate to both
	$f^{\circ N}=(A\circ f_0)^{\circ N}$ and $g^{\circ N}$, contradiction. We conclude that $D=1$. 

So, from now on, we may and do assume that 
	$\gcd(\delta,\deg(p))=1$. Similarly, $\gcd(\delta,\deg(q))=1$.
  
  	Let $\eta=\theta^{\circ 4}$, it will
  	be  more convenient to
  	work with the following identities:
  	$$f^{\circ 4n}\circ p=p\circ \eta;\ g^{\circ 4n}\circ q=q\circ \eta$$
  	Write $b=\deg(p)$ and $\tilde{b}=\deg(q)$.
  	By Lemma~\ref{lem:Inou}, there exist
  	positive integers $c\equiv \delta^{4n}$ modulo $b$
  	and $\tilde{c}\equiv \delta^{4n}$ modulo $\tilde{b}$
  	together with linear polynomials 
  	$\ell_1,\ell_2,\ell_3,\ell_4$
  	and non-constant polynomials $P$ and $Q$
  	such that the following groups of identities
  	hold:
  	\begin{itemize}
  	\item [(i)] $\ell_1\circ f^{\circ 4n}\circ \ell_1^{-1}=x^cP(x)^{b}$, 
  	$\ell_1\circ p\circ \ell_2^{-1}=x^b$,
  	and 
  	$\ell_2\circ \eta\circ \ell_2^{-1}=x^cP(x^b)$.
  	\item [(ii)]
  	$\ell_3\circ g^{\circ 4n}\circ \ell_3^{-1}=x^{\tc}Q(x)^{\tb}$, 
  	$\ell_3\circ q\circ \ell_4^{-1}=x^{\tb}$, and 
  	$\ell_4\circ \eta\circ \ell_4^{-1}=x^{\tc}Q(x^{\tb})$.
  	\end{itemize}
	We may assume that $P(0)\neq 0$ and $Q(0)\neq 0$ (since $f$ and $g$ are not linearly conjugate to the monomial $x^\delta$). Write
	$f_1:=\ell_1\circ f^{\circ 4}\circ \ell_1^{-1}$, $p_1:=\ell_1\circ p\circ \ell_2^{-1}=x^b$, $g_1:=\ell_3\circ g^{\circ 4}\circ \ell_3^{-1}$, and $q_1:=\ell_3\circ q\circ \ell_4^{-1}=x^{\tilde{b}}$.
	We
	now consider three cases.
	
	\textbf{Case 1:} $b=\tilde{b}=1$. Then we have that
	$f^{\circ 4n}$ and $g^{\circ 4n}$ are both linearly conjugate to 
	$\eta$. Hence there is a linear $L$ such that
	$f^{\circ 4n}=(L\circ g^4\circ L^{-1})^{\circ n}$. By Lemma~\ref{lem:f^4}
	and Lemma~\ref{lem:f^NLg^N}, there is
	$N\leq \frac{d^4}{2}$ such that
	$f^{\circ 4N}$ is linearly conjugate
	to $g^{\circ 4N}$. Hence $g^{\circ 4N}$ is semiconjugate
	to both $f^{\circ 4N}$ and
	$g^{\circ 4N}$,  contradicting
	the assumption that $(f,g)$ is a counter-example.
	
	\textbf{Case 2:} $b\geq 2$ and $\tilde{b}=1$. Then
	Lemma~\ref{lem:f^4 xsPn} gives that 
	$f_1 =x^{j}P_1(x)^b$
	for some non-constant $P_1(x)\in \C[x]\setminus x\C[x]$
	and positive integer $j$ coprime to $b$. Write $W(x)=x^jP_1(x^b)$ (which is not cyclic by part (e) of Lemma~\ref{lem:Gamma}), we have:
	$$f_1\circ p_1=p_1 \circ W.$$
	This implies:
	$$f_1^{\circ n}\circ p_1=p_1 \circ W^{\circ n}.$$
	Since $f_1^{\circ n}\circ p_1=p_1\circ (\ell_2\circ \eta\circ \ell_2^{-1})$, we have that $W^{\circ n}=\zeta(\ell_2\circ\eta\circ\ell_2^{-1})$ for some
	$b$-th root of unity $\zeta$. Since $\tilde{b}=1$,
	$g^{\circ 4n}$ is linearly conjugate to $\eta$. Hence
	there is a linear $\ell$ such that:
	$$W^{\circ n}=\zeta(\ell\circ g^{\circ 4n}\circ \ell^{-1}).$$
	By Lemma~\ref{lem:f^NLg^N},
    there is $N\leq \frac{d^4}{2}$ such that 
	$W^{\circ N}$ is linearly conjugate to $g^{\circ 4N}$. 
	Since $W^{\circ N}$ is semiconjugate to
	$f_1^{\circ N}$, so is
	$g^{\circ  4N}$. Since $f^{\circ 4N}$ is linearly 
	conjugate
	to $f_1^{\circ N}$, we have that $g^{\circ 4N}$ is semiconjugate to $f^{\circ 4N}$. Hence $g^{\circ 4N}$
	is semiconjugate to both $g^{\circ 4N}$ and $f^{\circ 4N}$,
	contradicting the assumption that $(f,g)$ is a counter-example.
	
	\textbf{Case 3:} $b\geq 2$ and $\tilde{b}\geq 2$. By using similar arguments
	as in Case 2, we have that $f_1=x^{j}P_1(x)^b$ and
	$g_1=x^{k}Q_1(x)^{\tilde{b}}$ for some non-constant
	$P_1,Q_1\in K[x]\setminus x K[x]$ and positive integers 
	$j,k$
	satisfying $\gcd(j,b)=\gcd(k,\tilde{b})=1$. 
	Write $W:=x^jP_1(x^b)$ and $V:=x^kQ_1(x^{\tilde{b}})$
	(which are not cyclic by part (e) of Lemma~\ref{lem:Gamma}), 
	we have:
	$$f_1\circ p_1=p_1\circ W\ \text{and}\ g_1\circ q_1=q_1\circ V.$$
	From $f_1^{\circ n}\circ p_1=p_1\circ (\ell_2\circ \eta
	\circ\ell_2^{-1})$
	and $g_1^{\circ n}\circ q_1=q_1\circ (\ell_4\circ\eta
	\circ\ell_4^{-1})$,
	we have:
	$$W^{\circ n}=\zeta(\ell_2\circ \eta\circ\ell_2^{-1})\ \text{and}\ 
	V^{\circ n}=\mu(\ell_4\circ \eta\circ \ell_4^{-1})$$
	for some $b$-th root of unity $\zeta$ and $\tilde{b}$-th
	root of unity $\mu$. Hence there are linear polynomials
	$L$ and $L_1$ such that:
	$$W^{\circ n}=L\circ (L_1\circ V^{\circ n}\circ L_1^{-1}).$$
	By Lemma~\ref{lem:f^NLg^N},
	there exists $N\leq \frac{d^4}{2}$ such that 
	$W^{\circ N}$ is linearly conjugate to $V^{\circ N}$. Since
	$W^{\circ N}$ (respectively $V^{\circ N}$) is
	semiconjugate to
	 $f_1^{\circ N}$ (respectively $g_1^{\circ N}$),
	and
	$f_1^{\circ N}$ (respectively $g_1^{\circ N}$) is linearly
	conjugate to $f^{\circ 4N}$ (respectively $g^{\circ 4N}$),
	we conclude that there is a common polynomial
	semiconjugate to both 
	$f^{\circ 4N}$ and $g^{\circ 4N}$. This contradicts the
	assumption that $(f,g)$ is a counter-example.
	\end{proof}
	
	\begin{proof}[Proof of Corollary~\ref{cor:new semiconjugate}]
	   By Theorem~\ref{thm:semiconjugate}, we define
	   $c_1(d,2)=2d^4$ for every integer $d\geq 2$. 
	   Let $n\geq 3$,
	   and assume that we have determined $c_1(d,n)$ for all
	   smaller values of $n$ and all $d\geq 2$. 
	   By the induction hypothesis, there exist positive
	   integers 
	   $N_1\leq c_1(d,n-1)$  
	   and $\theta_1\in K[x]$
	   such that $\theta_1$ is semiconjugate to
	   $f_i^{\circ N_1}$
	   for $1\leq i\leq n-1$. 
	   
	   We have 
	   $\theta_1\approx f_1^{\circ N_1}\approx f_n^{\circ N_1}$. By 
	   Theorem~\ref{thm:semiconjugate}, there exist 
	   $N_2\leq 2d^{4N_1}$ and a polynomial 
	   $\theta$ that is semiconjugate
	   to both $\theta_1^{\circ N_2}$
	   and $f_n^{\circ N_1N_2}$. We have that
	   $\theta$ is semiconjugate
	   to $f_i^{\circ N_1N_2}$
	   for every $1\leq i\leq n$. Define 
	   $c_1(d,n)=c_1(d,n-1)2d^{4c_1(d,n-1)}$
	   so that $N:=N_1N_2\leq c_1(d,n)$; this finishes the
	   proof of Corollary~\ref{cor:new semiconjugate}.
	\end{proof}
	
	\section{The Medvedev-Scanlon classification of periodic subvarieties}\label{sec:Medvedev-Scanlon}
	Let $n\in\N$. The main results of the paper \cite{MedSca} 
	provide
	a description of periodic subvarieties of
	$\bA^n_K$ under coordinate-wise self-maps
	of the form
	$$(x_1,\ldots,x_n)\mapsto (f_1(x_1),\ldots,f_n(x_n))$$
	where $f_1,\ldots,f_n\in K[x]$ are non-constant.
	The main reference is the original paper
	by Medvedev-Scanlon \cite{MedSca} together
	with further remarks in \cite{Nguyen-IMRN2015}
	and \cite{DBHC}.
	We start with
	the following:
	\begin{proposition}\label{prop:MS 3 types}
		Let $L_1,\ldots,L_r\in K[x]$ be linear polynomials,
		let $M_1,\ldots,M_s\in K[x]$ be polynomials
		of degree at least 2 such that each $M_i$ is
		linearly conjugate to $x^{\deg(M_i)}$ or $\pm T_{\deg(M_i)}(x)$ for $1\leq i\leq s$, and let
		$f_1,\ldots,f_t\in K[x]$ be disintegrated polynomials.
		Write $n=r+s+t$ and let
		$\phi$ be the coordinate-wise self-map
		of $\bA^n$ induced by the polynomials $L_i$'s, $M_i$'s,
		and $f_i$'s. Then every irreducible $\phi$-periodic subvariety $V$
		has the form $V_1\times V_2\times V_3$
		where $V_1$, $V_2$, and $V_3$ respectively
		are periodic under the coordiniate-wise self-maps
		$L_1\times\cdots\times L_r$,
		$M_1\times\cdots\times M_s$,
		and $f_1\times\cdots\times f_t$.
	\end{proposition}
	\begin{proof}
	This is a consequence of \cite[Theorem~2.30]{MedSca}.
	\end{proof}
	
	We now study periodic subvarieties 
	under the coordinate-wise self-map
	$f_1\times\cdots\times f_n$ where each $f_i\in K[x]$
	is disintegrated:
	\begin{proposition}\label{prop:pi_ij}
		Let $n\geq 2$, let $f_1,\ldots,f_n\in K[x]$ be disintegrated polynomials, and let $\phi=f_1\times\cdots\times f_n$
		be the induced coordinate-wise self-map
		of $\bA^n_K$. Then every irreducible $\phi$-periodic 
		subvariety $V$ of $\bA^n_K$
			is an irreducible component of 
			 $$\bigcap_{1\leq i<j\leq n}\pi_{ij}^{-1}(\pi_{ij}(Z)),$$
			 where $\pi_{ij}$ denotes the projection
			 from $\bA^n_K$
			 to the $(i,j)$-factor $\bA^2_K$.
	\end{proposition}
	\begin{proof}
		See Proposition~2.21 and Fact~2.25 in \cite{MedSca}.
	\end{proof}
	
	 Recall the equivalence relation $\approx$ in 
		Definition~\ref{def:equiv}.	Note that
		$f\approx g$ if and only if 
		the coordinate-wise self-map $f\times g$
		of $\bA^2$ admits an irreducible periodic curve
		whose projection to each factor $\bA^1$
		is non-constant (see \cite[Corollary~2.35]{MedSca}).
		Proposition~\ref{prop:pi_ij} implies
		the following:
	\begin{proposition}\label{prop:MS equivalence classes}
		Let $s,n_1,\ldots,n_s$
		be positive integers. For $1\leq i\leq s$
		and $1\leq j\leq n_i$, let $f_{i,j}\in K[x]$
		be a disintegrated polynomial such that the $f_{i,j}$'s
		form exactly $s$ equivalence classes each of
		which being $(f_{i,j})_{1\leq j\leq n_i}$
		for $1\leq i\leq s$. For $1\leq i\leq s$,
		let $\phi_i=f_{i,1}\times\cdots\times f_{i,n_i}$
		be the induced coordinate-wise self-map
		of $\bA^{n_i}$. Write $n=n_1+\ldots+n_s$,
		identify $\bA^n=\bA^{n_1}\times\cdots\times \bA^{n_s}$,
		and let $\phi=\phi_1\times\cdots\times\phi_s$
		be the induced coordinate-wise self-map
		of $\bA^n$. The following hold:
		\begin{itemize}
			\item [(i)] Every irreducible $\phi$-periodic
			subvariety $V$ of $\bA^n$ has the form
			$V_1\times\cdots\times V_s$ where each $V_i$
			is a $\phi_i$-periodic
			subvariety of $\bA^{n_i}$ for $1\leq i\leq s$.
		
			\item [(ii)] Let $1\leq i\leq s$. There exists 
			$N_i\in\N$, non-constant polynomials
			$\theta_i$ and $p_{i,j}$ for $1\leq j\leq n_i$
			such that $f_{i,j}^{\circ N_i}\circ p_{i,j}=
			p_{i,j}\circ \theta_i$. In other words, we
			 have the commutative diagram (with $p_i:=p_{i,1}\times\cdots\times p_{i,n_i}$):
			 $$\begin{diagram}
				 		\node{\bA^{n_i}}\arrow{s,t}{p_i}\arrow{e,t}{\theta_i\times\cdots\times\theta_i}
				 		\node{\bA^{n_i}}\arrow{s,r}{p_i}\\
				 		\node{\bA^{n_i}}\arrow{e,b}{\phi_i^{\circ N_i}}\node{\bA^{n_i}}
				 \end{diagram}$$
		\end{itemize}
	\end{proposition}
	\begin{proof}
		Part (a) is a consequence of 
		\cite[Theorem~2.30]{MedSca}. Part (b) (which was
		alluded in Section~\ref{sec:semiconjugacy})
		is proved in \cite[Section~7.1]{DBHC}.
	\end{proof}

\begin{remark}
\label{same degree remark}	
We note that (according to part~(ii) of Proposition~\ref{prop:MS equivalence classes}), if there exists a periodic $\phi$-subvariety $V$ as in Proposition~\ref{prop:MS equivalence classes}, then $\deg(f_{i,1})=\cdots =\deg(f_{i,n_i})$ for each $i$.
\end{remark}

	Proposition~\ref{prop:MS equivalence classes} shows that
	it suffices to treat self-maps of $\bA^n$ 
	of the form
	$f\times\cdots\times f$ of where $f$ is disintegrated.
	Let $x_1,\ldots,x_n$ denote the coordinate functions
	of each factor $\bA^1_K$ of $\bA^n_K$.
	We have:
	\begin{proposition}\label{prop:MS diagonal}
		Let $n\in\N$ and $f\in K[x]$ be disintegrated. Let
		$\phi:=f\times\cdots\times f$ be the 
		induced coordinate-wise self-map of $\bA^n_K$. Then
		every irreducible $\phi$-periodic subvariety $V$
		of $\bA^n_K$ is defined by a collection of equations of 
		the following forms:
		\begin{itemize}
			\item [(i)] $x_i=\zeta$ for some $1\leq i\leq n$
			 and some $\zeta\in K$ that is $f$-periodic.
			 \item [(ii)] $x_i=g(x_j)$
			 for some $1\leq i\neq j\leq n$ and 
			 some non-constant $g\in K[x]$
			 that commutes with an iterate of $f$.	
		\end{itemize}
	\end{proposition}
	\begin{proof}
		See \cite[Theorem~6.24]{MedSca}.
	\end{proof}

	\section{Proof of Theorem~\ref{thm:intro1}}\label{sec:proof main}
	Obviously, we can define $c(d,1)=1$ for
	every $d\geq 2$ since $\bA^1$ is the only 
	periodic subvariety satisfying the 
	condition of the theorem when $n=1$. Let $n\geq 2$, assume
	that $c(d,n)$ has been defined for all smaller values of
	$n$ and for all $d\geq 2$. Recall the
	constant $c_1(d,n)$ in Theorem~\ref{cor:new semiconjugate},
	define:
	$$c(d,n):=\max\left\{c(d,n-1)^{n-1},\frac{d^{c_1(d,n)}}{2}\right\}.$$
	Under the equivalence
	relation $\approx$, assume that $f_1,\ldots,f_n$
	belong to exactly $s$ equivalence classes
	whose sizes are $n_1,\ldots,n_s$. We consider two cases:
	
	\textbf{Case 1:} $s\geq 2$. By Proposition~\ref{prop:MS equivalence classes}, after rearranging the
	factors $\bA^1$ of $\bA^n$, every irreducible
	$\varphi$-periodic
	subvariety $V$ of $\bA^n$ has the form
	$V_1\times\cdots\times V_s$. By the induction hypothesis,
	the period of each $V_i$ is at most
	$c(d,n_i)$. Hence the period of $V$ is at most:
	$$c(d,n_1)\ldots c(d,n_s)\leq c(d,n-1)^{n-1}\leq c(d,n).$$
	
	\textbf{Case 2:} $s=1$. By Theorem~\ref{cor:new semiconjugate}, there exist a positive integer $N\leq c_1(d,n)$
	and non-constant polynomials $\theta,p_1,\ldots,p_n\in K[x]$ such that $f_i^{\circ N}\circ p_i=p_i\circ \theta$
	for $1\leq i\leq n$. Let $\phi:=\theta\times\cdots\times\theta$ and $\rho:=p_1\times\cdots\times p_n$ be the
	induced coordinate-wise self-maps of
	$\bA^n$. We have the commutative diagram:
	
	$$\begin{diagram}
				 		\node{\bA^n}\arrow{s,t}{\rho}\arrow{e,t}{\phi}
				 		\node{\bA^n}\arrow{s,r}{\rho}\\
				 		\node{\bA^n}\arrow{e,b}{\varphi^{\circ N}}\node{\bA^n}
				 \end{diagram}$$
	
	Since $V$ is periodic under $\varphi^{\circ N}$, 
	some irreducible component $\tilde{V}$
	of $\rho^{-1}(V)$ is periodic under
	$\phi$. Since the projection from $V$ to each factor
	$\bA^1$ is non-constant, so is the projection
	from $\tilde{V}$ to each factor $\bA^1$. Therefore,
	by Proposition~\ref{prop:MS diagonal},
	$\tilde{V}$ is defined by equations of the form
	$x_i=g(x_j)$ for some $1\leq i\leq j$ and 
	some non-constant $g\in K[x]$ commuting with an iterate of 
	$\theta$. By Lemma~\ref{lem:commute everything},
	there is $N_1\leq \frac{\deg(\theta)}{2}=\frac{d^N}{2}$
	such that $\theta^{\circ N_1}$
	commutes with $g$. Hence the $\phi$-period
	of $\tilde{V}$
	is at most $N_1$. Consequently, the
	$\varphi$-period of $V$ is at most:
	$$N_1\leq \frac{d^N}{2}\leq \frac{d^{c_1(d,n)}}{2}\leq c(d,n).$$
	This finishes the proof of Theorem~\ref{thm:intro1}.

	\section{Proof of Theorem~\ref{thm:DML}}\label{sec:proof DML}
	\subsection{A simple reduction}
Obviously, it suffices to prove Theorem~\ref{thm:DML} when $C$ is irreducible. As shown in the proof of \cite[Theorem~1.4]{BGKT} (see also the proof of Conjecture~5.10.0.17 as a consequence of Conjecture~5.10.0.18 in \cite{BGT-book}), it suffices to prove Theorem~\ref{thm:DML} when $m=2$. Also, the theorem (in the case $m=2$) follows immediately if the curve $C$ does not project dominantly onto both axes of $\bA^2$; hence, in particular, we may assume that the polynomials $F_i$'s are not constant. In addition, write $\alpha:=(\alpha_1,\alpha_2)$ and if $\alpha_i$ is preperiodic under the action of $F_i$
for some $i\in\{1,2\}$, then the result also follows easily (see \cite[Proposition~3.1.2.9]{BGT-book}). Furthermore, if $C$ is periodic under the action of $\Phi$, then Theorem~\ref{thm:DML} follows readily. Therefore, all we need to prove is the following theorem.

\begin{theorem}
\label{thm:the result}
Let $F_1(x), F_2(x)\in K[x]$ be non-constant and let $\alpha_1,\alpha_2\in K$ such that for each $i=1,2$, we have that $\alpha_i$ is not $F_i$-preperiodic. Let $\alpha:=(\alpha_1,\alpha_2)$, let $\Phi:\bA^2_K\lra \bA^2_K$ be the endomorphism given by $\Phi(x,y):=(F_1(x),F_2(y))$, and let $C\subset \bA^2_K$ be an irreducible curve which is not $\Phi$-periodic. Assume that the projection from $C$ to each factor
$\bA^1_K$ is not constant. Then the set $S:=\{n\in\N\colon \Phi^n(\alpha)\in C\}$ is finite.
\end{theorem}

The rest of this section is dedicated to the proof of Theorem~\ref{thm:the result}. Note that Theorem~\ref{thm:the result} 
in the case $K=\Qbar$ follows from recent work of Xie \cite{Xie-2015}.
In the next subsection, we explain how to combine Theorem~\ref{thm:intro1} together with standard specialization arguments
to settle the case of general $K$.

\subsection{Proof of Theorem~\ref{thm:the result}}
\label{subsec:proof the result}

We work under the hypotheses of Theorem~\ref{thm:the result}.

Let $\mathcal{K}\subset K$ be a finitely generated field over $\Qbar$ such that $\alpha_1,\alpha_2\in \mathcal{K}$, the coefficients of $F_1$ and of $F_2$ are in $\mathcal{K}$, and also the curve $C$ is defined by $G(x,y)=0$ for
$G(x,y)\in \mathcal{K}[x,y]$. We argue by induction on the transcendence degree $e$ of $\mathcal{K}/\Qbar$. The case $e=0$ has been obtained by Xie \cite[Theorem~0.3]{Xie-2015}. So, we may assume $\mathcal{K}$ is a function field of transcendence degree $1$ over some subfield $E$, and moreover, we may assume that Theorem~\ref{thm:the result} holds if $F_1$, $F_2$, $\alpha_1$, $\alpha_2$ and $C$ are all defined over $\Ebar$. 
If both $F_1$ and $F_2$ are linear polynomials, then $\Phi$ is an automorphism of $\bA^2$ and in this case, the Dynamical Mordell-Lang Conjecture is known to hold as proven by Bell \cite{Bell}. So, from now on, assume that $\deg(F_2)\ge 2$. Assume that the set $\{m\in\N: \Phi^{\circ m}(\alpha)\in C\}$
is infinite and we will arrive at a contradiction.

At the expense of replacing $E$ by its algebraic closure inside $\mathcal{K}$, we may assume that $\mathcal{K}$ is the function field of a smooth geometrically irreducible projective curve 
$X$ defined over $E$. Then for all but finitely many point $\fp$ of $X(\Ebar)$ 
we can specialize both $F_1$ and $F_2$ at those points and the corresponding specialization $F_{1,\fp},F_{2,\fp}\in \Ebar[x]$  satisfy the following properties:
\begin{itemize}
\item[(i)] $\deg(F_{i,\fp})=\deg(F_i)$ for each $i=1,2$;
\item[(ii)] for $i=1,2$, if $F_i$ is disintegrated, then $F_{i,\fp}$ is disintegrated.
\end{itemize}
Condition (i) is verified by all places of good reduction for the polynomials $F_1$ and $F_2$. Condition~(ii) is also satisfied by all but finitely many places in $X(\Ebar)$ according to \cite[Proposition~7.8]{BGKT}. In addition to conditions (i)-(ii), the specialization at all but finitely many points of $X(\Ebar)$ satisfies the following properties:
\begin{itemize}
\item[(iii)] each coefficient of $G$ is integral at $\fp$; and
\item[(iv)] the curve $C_\fp$ which is the zero locus of $G_{\fp}(X,Y)=0$ (where $G_{\fp}$ is the polynomial obtained by reducing each coefficient of $G$ modulo $\fp$) is also geometrically irreducible (over $\Ebar$). In addition, the
projection from $C_\fp$ to each factor 
$\bA^1_{\Ebar}$ is not constant.
\end{itemize} 
Clearly, condition (iii) is satisfied by all but finitely many places in $X(\Ebar)$. The same is true regarding the first
assertion of condition~(iv) (this follows from the Bertini-Noether theorem, see \cite[pp.~170]{FriedJarden}). Since
$\deg_y(G)=\deg_y(G_\fp)$ and $\deg_x(G)=\deg_x(G_\fp)$
for all but finitely many $\fp$, the
second assertion in condition~(iv) also holds.
Finally, according to \cite[Proposition~6.2]{GTZieve}, there are infinitely many $\fp\in X(\Ebar)$ such that:
\begin{itemize}
\item[(v)] the reduction $\alpha_{2,\fp}$ of $\alpha_2$ modulo $\fp$ is not preperiodic under  $F_{2,\fp}$.
\end{itemize}
We denote by $T$ the infinite set of places $\fp$ satisfying conditions (i)-(v). Then for each such $\fp\in T$, the orbit of $(\alpha_{1,\fp}, \alpha_{2,\fp})$ under $F_{1,\fp}\times F_{2,\fp}$ intersects the curve $C_{\fp}$ in infinitely many points. So, by the inductive hypothesis, we know that $C_{\fp}$ is periodic under $F_{1,\fp}\times F_{2,\fp}$. Since
the projection from $C_\fp$ to 
each $\bA^1_{\Ebar}$ is non-constant and 
$\deg(F_2)\geq 2$,
Proposition~\ref{prop:MS 3 types} and 
Proposition~\ref{prop:MS equivalence classes} 
show that it
suffices to consider the following two cases.

{\bf Case 1.}  $F_1$ and $F_2$ are disintegrated polynomials
of the same degree $\delta$ (see Remark~\ref{same degree remark} applied to the polynomials $F_{1,\fp}$ and $F_{2,\fp}$, and also note that $\deg(F_1)=\deg(F_{1,\fp})$ and $\deg(F_2)=\deg(F_{2,\fp})$ for $\fp\in T$). By Theorem~\ref{thm:intro1},
there is a constant $N$ depending only on $\delta$
such that $(F_{1,\fp}\times F_{2,\fp})^{\circ N}(C_\fp)=C_\fp$
for every $\fp\in T$. Since $T$ is infinite, we have
$\Phi^{\circ N}(C)=C$. This gives that $C$ is $\Phi$-periodic, contradiction.

{\bf Case 2.} Both $F_1$ and $F_2$ are not disintegrated.
Write $\delta_i=\deg(F_i)$ for $i=1,2$. It
is more convenient to regard $\Phi$
as a self-map of $(\bP^1)^2$ and 
replace $C$ by its Zariski closure in $(\bP^1)^2$. 
For $i=1,2$, define 
$\nu_i(x)=x$ if $F_i$ is linearly conjugate to a monomial
and $\nu_i(x)=x+\frac{1}{x}$ if $F_i$ is linearly conjugate
to $\pm T_{\delta_i}(x)$. Let $\nu=\nu_1\times\nu_2$ be
the corresponding coordinate-wise self-map of
$(\bP^1)^2$. We have the commutative diagram:
	$$\begin{diagram}
	   \node{(\bP^1)^2}\arrow{s,t}{\nu}\arrow{e,t}{\tilde{\Phi}}
				 		\node{(\bP^1)^2}\arrow{s,r}{\nu}\\
				 		\node{(\bP^1)^2}\arrow{e,b}{\Phi}\node{(\bP^1)^2}
				 \end{diagram}$$
where $\tilde{\Phi}$ has the form $(x,y)\mapsto (\pm x^{\delta_1},\pm x^{\delta_2})$. Pick $\beta\in \nu^{-1}(\alpha)$, then there is an irreducible component
$\tilde{C}$ of $\nu^{-1}(C)$ having an infinite intersection
with the $\tilde{\Phi}$-orbit of $\beta$.
In this case, it is well-known 
that $\tilde{C}$ is  $\tilde{\Phi}$-periodic (see either \cite[Theorem~1.8]{GT09}, or 
\cite[Theorem~1.3]{BGT} with the remark that $\tilde{\Phi}$ induces an \'etale
endomorphism of $\Gm^2$). Therefore $C$ is $\Phi$-periodic, contradiction. This concludes the proof of Theorem~\ref{thm:the result} and therefore the proof of Theorem~\ref{thm:DML}.
	
	\bibliographystyle{amsalpha}
	\bibliography{Period-Oct4}

\end{document}